\documentclass[a4paper, 11pt]{amsart}
\usepackage{amssymb,amsmath, stmaryrd}
\usepackage[dvipdfmx]{graphicx}
\usepackage[all]{xy}
\usepackage{bm}
\usepackage[top=25truemm,bottom=25truemm,left=25truemm,right=25truemm]{geometry}
\usepackage{color}
\usepackage{ulem}

\makeatletter
\@namedef{subjclassname@2010}{%
  \textup{2010} Mathematics Subject Classification}
\makeatother

\newcommand{\A}{\mathbb{A}}

\newcommand{\N}{\mathbb{N}}

\newcommand{\R}{\mathbb{R}}
\newcommand{\Z}{\mathbb{Z}}

\newcommand{\rank}{\operatorname{rank}}

\newcommand{\Aut}{\operatorname{Aut}}
\newcommand{\Gap}{\operatorname{Gap}}

\newcommand\ord{\operatorname{ord}}

\newtheorem{thm}{Theorem}[section]
\newtheorem{prop}[thm]{Proposition}
\newtheorem{lem}[thm]{Lemma}
\newtheorem{cor}[thm]{Corollary}

\newtheorem{example}[thm]{Example}
\newtheorem{question}[thm]{Question}

\newtheorem{ex}[thm]{Example}

\begin{document}
\title[Quasi-factorially closed subalgebras of Laurent polynomial rings]{Quasi-factorially closed subalgebras of Laurent polynomial rings}
\author{Shinya Kumashiro and Takanori Nagamine}
\address[S. Kumashiro]
{Department of Mathematics, Osaka Institute of Technology, 5-16-1 Omiya, asahi-ku, Osaka, 535-8585, Japan}
\email{shinya.kumashiro@oit.ac.jp}
\address[T. Nagamine]
{Department of Mathematics, College of Science and Technology\\
Nihon University, 1-8-14 Kanda-Surugadai, Chiyoda-ku, Tokyo 101-8308, Japan}
\email{nagamine.takanori@nihon-u.ac.jp}
\date{\today}
\thanks{The work of the first author was supported by JSPS KAKENHI Grant Number JP24K16909. The work of the second author was supported by JSPS KAKENHI Grant Number JP25K17239.}

\begin{abstract}
Let $R$ be a domain and $B=R[x_1^{\pm1},\ldots,x_n^{\pm1}]$ the Laurent polynomial ring over $R$. In this paper we study pre-factorially closed (pfc) and quasi-factorially closed (qfc) $R$-subalgebras of $B$, which generalize the notion of factorially closed subalgebras. 

We first establish a localization criterion for the qfc property. Using this criterion, we investigate monoid algebras $A=R[M]$ associated with submonoids $M\subset \mathbb{Z}^n$. We prove that $R[M]$ is qfc in $B$ if and only if the group generated by $M$ is a direct summand of $\mathbb{Z}^n$. This provides a complete characterization of the qfc property in terms of the lattice structure of the associated group. As a consequence, when $n=1$ and $M\subset\mathbb{N}$, the algebra $R[M]$ is qfc in $B$ precisely when $M$ is a numerical semigroup. 
For a general $R$-subalgebra $A\subset B$, we introduce an invariant $\mathrm{Gap}(A)$. We show that if $\mathrm{Gap}(A)$ is finite, then $A$ is qfc in $B$. 

Moreover, we clarify how the pfc and qfc conditions are related to other notions that naturally appear for subalgebras, such as retracts, being algebraically closed in $B$, and normality. 
\end{abstract}
\maketitle

\setcounter{section}{0}

\section{Introduction}
Throughout the paper, all rings are commutative with unity, any domain is understood to be an integral domain and $\N$ contains $0$.

Let $R$ be a domain and let $A\subset B$ be $R$-algebras.
We write $R^*$ for the group of units of $R$ and $Q(R)$ for the quotient field of $R$.
For an integer $n\ge 0$, $R^{[n]}$ (resp. $R^{[\pm n]}$) denotes the polynomial ring (resp. Laurent polynomial ring) in $n$ variables over $R$.

Recall that $A$ is {\bf factorially closed}, or {\bf fc}, in $B$ if for $a,b\in B\setminus\{0\}$ the condition that $ab\in A$ implies $a,b\in A$. By the definition, it is easy to see that if $B$ is a UFD, then so is $A$ (Proposition \ref{ac and fc}). On the other hand, if $A^*\subsetneq B^*$, then $A$ cannot be fc in $B$ (\cite[Lemma 3 (4)]{CGM15}). With this observation, it is natural to generalize the notion of fc, and actually the following two notions are introduced in the literatures:

\begin{enumerate}
\item $A$ is {\bf pre-factorially closed}, or {\bf pfc}, in $B$ if for $a,b\in B\setminus\{0\}$ the condition $ab\in A$ implies $ua,u^{-1}b\in A$ for some $u\in B^*$.
\item $A$ is {\bf quasi-factorally closed}, or {\bf qfc}, in $B$ if for $a,b\in B\setminus\{0\}$ the condition $ab\in A$ implies $ua,vb\in A$ for some $u,v\in B^*$. 
\end{enumerate}

Clearly, implications ``fc''$\implies$``pfc''$\implies$``qfc'' hold. We note that $A$ is fc in $B$ if and only if $A$ is qfc (or pfc) in $B$ and $A^*=B^*$ (Lemma \ref{fc}). Thus, qfc and pfc are natural generalizations of the fc property in the study of rings $B$ with many units.

The pfc property was introduced by Cohn \cite{Coh68}\footnote{Cohn refers to the notion we call fc in this paper as {\bf inert}. However, in the footnote on page 259 of \cite{Coh68}, he states that an inert subalgebra should actually be called a {\bf strongly inert} subalgebra and that the notion of inert should be the pfc property as defined above. Since the terms inert and strongly inert are used inconsistently in the literature, we use the terminology pfc to avoid confusion.} and studied in \cite{AAZ92,AC22,Coh68}. 
The notion of qfc was introduced by Chakraborty, Gurjar and Miyanishi \cite{CGM15}, and further studied in \cite{CGM15,CGM19,Mon21}. In the setting of polynomial rings and formal power series rings, Chakraborty, Gurjar and Miyanishi \cite{CGM15} obtained strong structure theorems for fc and qfc subalgebras in low dimensions:

\begin{thm} 
Let $k$ be an algebraically closed field of characteristic zero. Then the following assertions hold true.
\begin{enumerate}
\item
If $n\leq3$, then every fc subalgebra of $k^{[n]}$ is isomorphic to $k^{[d]}$ for some $d\geq0$ {\rm(\cite[Theorem 1]{CGM15})}. 
\item
Let $k\subsetneq A \subsetneq B:=k\llbracket x,y\rrbracket$ be a qfc subalgebra of $B$. If $A$ is a noetherian complete local ring, then $A=k\llbracket u\rrbracket$ {\rm(\cite[Theorem 7]{CGM15})}.
\item
Let $k\subset A \subset B:=k\llbracket x,y,z\rrbracket$ be a qfc subalgebra of $B$. If $A$ is a 2-dimensional noetherian complete local ring, then $A=k\llbracket u,v\rrbracket$ {\rm(\cite[Theorem 8]{CGM15})}.
\end{enumerate}
\end{thm}

As an analogue of \cite{CGM15}, in this paper we study pfc and qfc $R$-subalgebras of the Laurent polynomial ring
\[
B=R[x_1^{\pm1},\ldots,x_n^{\pm1}]
\]
over a domain $R$. This is a natural setting where $B$ is a graded ring and $B$ has many units, hence the distinction between fc and pfc/qfc becomes essential. Our guiding problem is the following. 

\begin{question}
{\rm 
What are the pfc/qfc $R$-subalgebras of $R^{[\pm n]}$? 
}
\end{question}

We first establish a localization criterion for the qfc property.
Let $A$ be an $R$-subalgebra of the Laurent polynomial ring $B$ and set $S:=A\cap B^*$.
Then $A$ is qfc in $B$ if and only if $S^{-1}A$ is qfc in $B$ (Proposition \ref{new criterion}).
This criterion is particularly useful when $A$ is a monoid algebra. 
Let $M$ be a submonoid of $\Z^n$. Assume that
\[
A=R[M]:=R[\bm{x}^{\bm{a}} \mid \bm{a}\in M],
\]
where $\bm{x}^{\bm{a}}=x_1^{a_1}\cdots x_n^{a_n}$ for $\bm{a}=(a_1,\ldots,a_n)\in M$.
Then we observe that 
\[
S^{-1}A = R[\langle M\rangle],
\]
where $\langle M\rangle$ denotes the group generated by $M$.
Thus, for monoid algebras, the qfc problem for $R[M]$ can be reduced to the qfc problem for the group algebra $R[\langle M\rangle]$. Our first main result gives a complete criterion in terms of the lattice structure of $\langle M\rangle$.

\begin{thm} \label{higher numerical semigroup}
Let $B\cong R^{[\pm n]}$ be the Laurent polynomial ring in $n$ variables over $R$, $M$ be a submonoid of $\Z^n$ and $\langle M\rangle$ be the group generated by $M$. 
Then the following conditions are equivalent. 
	\begin{enumerate}
	\item $R[M]$ is qfc in $B$. 
	\item $R[\langle M\rangle]$ is qfc in $B$.
	\item $\langle M\rangle$ is a direct summand of $\Z^n$. 
	\item ${\rm Cone}(\langle M\rangle)\cap\Z^n=\langle M\rangle$. 
	\end{enumerate}
\end{thm}

It would be worth comparing Theorem \ref{higher numerical semigroup} with the classical normality criterion for affine semigroup rings.
Recall that when $R$ is a field and $M$ is finitely generated, the monoid algebra $R[M]$ is normal if and only if
\[
{\rm Cone}(M)\cap\langle M\rangle = M.
\]
Thus, the qfc property differs slightly from normality (see, for example, Examples \ref{normal-not pfc} and \ref{non-normal}). In view of this difference, we clarify how the pfc and qfc conditions are related to other notions that naturally appear for subalgebras, such as retracts, being algebraically closed in $B$, and normality. For these definitions, see the beginning of Section 2. 
We obtain the following chain of implications for $R[M]$.
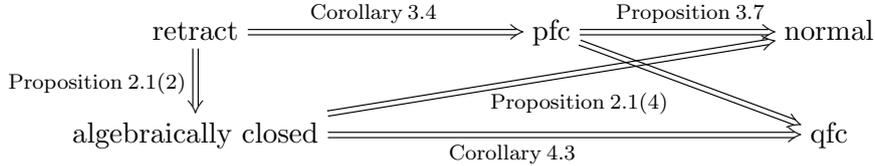
\[
\xymatrix@C=3em@R=2em{
\text{retract} \ar@{=>}[d]_{\text{Proposition\:}\ref{ac and fc}(2)} \ar@{=>}[rr]^{\text{Corollary\:}\ref{ret-pfc}} &&
\text{pfc} \ar@{=>}[rrd] \ar@{=>}[rr]^{\text{Proposition\:}\ref{pfc normal}} && 
\text{normal}\\
\text{algebraically closed} \ar@{=>}[rrrr]_{\text{Corollary\:}\ref{ac}}  \ar@{=>}[urrrr]_{\hspace{3em}\text{Proposition\:}\ref{ac and fc}(4)} & &&&
\text{qfc} 
}
\]
We note that none of the converses holds in general (see Examples \ref{pfc-not ret}, \ref{normal-not pfc} and \ref{non-normal}).

We further prove that when $M$ is a submonoid of $\N$, $R[M]$ is qfc in $B$ if and only if $M$ is a numerical semigroup (Corollary \ref{numerical semigroup}). Here, recall that a submonoid $M\subset\N$ is called a {\bf numerical semigroup} if $\N\setminus M$ is finite. 

\medskip
We next turn to the general case of arbitrary $R$-subalgebras of the Laurent polynomial ring $B$.
For a monoid algebra $R[M]$, the key role was played by the set
\[
{\rm Cone}(M)\cap\Z^n\setminus M.
\]
Motivated by this observation, in Subsection \ref{sec32} we introduce an invariant ${\rm Gap}(A)$ for an arbitrary $R$-subalgebra $A\subset B$.
When $A=R[M]$ is a monoid algebra, the invariant ${\rm Gap}(A)$ coincides with ${\rm Cone}(M)\cap\Z^n\setminus M$. For the detailed definition, see before Lemma \ref{inner}. 
Using this invariant, we obtain the following sufficient condition for the qfc property.

\begin{thm} \label{fin-gap}
Let $B\cong R^{[\pm n]}$ be the Laurent polynomial ring in $n$ variables over $R$ and $A$ be an $R$-subalgebra of $B$. If $\Gap(A)$ is a finite set, then $A$ is qfc in $B$.
\end{thm}

The finiteness of ${\rm Gap}(A)$ in Theorem \ref{fin-gap} is only a sufficient condition and is not necessary in general (Example \ref{infinite}).
More strikingly, we construct a qfc $R$-subalgebra $A\subset B$ that contains no monomials (Example \ref{niceex}).

\medskip
This paper is organized as follows. In Section 2, we establish basic properties of pfc and qfc subalgebras. We then give a localization criterion for the qfc property in Proposition \ref{new criterion}. 

In Section 3, we characterize pfc and qfc subalgebras of Laurent polynomial rings in certain cases. In particular, in Subsection 3.2 we introduce a new invariant for an $R$-subalgebra $A$ of $B$, called the gap set ${\rm Gap}(A)$. We then prove Theorem \ref{fin-gap}. 

In Section 4, we provide a characterization of the qfc property in the case where the subalgebra is generated by monic monomials. 
Translating the arguments of this section into the language of monoid algebras, we obtain Theorem \ref{higher numerical semigroup}. 

In Section 5, we determine pfc and qfc subalgebras in the one-variable case without assuming that the subalgebra is generated by monic monomials. An example of a qfc subalgebra that contains no monomials is also given (Example \ref{niceex}).

Finally, in Section 6, we answer the question posed in \cite{CGM15}: {\it Is any fc subalgebra of a PID also a PID?} More precisely, we show that the problem had essentially already been answered in the negative by Cohn in 1968 \cite{Coh68}. By presenting a precise formulation of his argument and combining it with a method developed in \cite{CGM19}, we obtain a complete solution to the question.
\section{Preliminaries}

Throughout this section, $R$ denotes a domain and let $A\subset B$ be $R$-domains. $A$ is called an $R$-{\bf retract} of $B$ if there are $R$-algebra homomorphisms $\varphi:B\to A$ and $\psi:A\to B$ such that $\varphi\circ\psi={\rm id}_A$.  
An element $b\in B$ is said to be {\bf algebraic} over $A$ if $f(b)=0$ for some $0\not=f\in A^{[1]}$. $A$ is {\bf algebraically closed in $B$}, or {\bf ac},  if every algebraic element of $B$ over $A$ belongs to $A$. The following are basic properties of the ac, fc and retract conditions for $R$-subalgebras.

\begin{prop} \label{ac and fc}
The following assertions hold. 
\begin{enumerate} 
\item Every fc $R$-subalgebra of $B$ is ac in $B$.
\item Every $R$-retract of $B$ is ac in $B$. 
\item Suppose that $B$ is a UFD. Then, every fc $R$-subalgebra of $B$ is a UFD.
\item Suppose that $B$ is normal. Then, every ac  $R$-subalgebra of $B$ is normal.
\end{enumerate}
\end{prop}
\begin{proof}
{\bf (1)}, {\bf (3)} Omitted. {\bf (2)} See \cite[Lemma 1.3]{Cos77}.

\medskip
{\bf  (4)} Let $A$ be an ac $R$-subalgeba of $B$. Let $\alpha\in Q(A)$ be integral over $A$. Then $\alpha \in Q(B)$ and is integral over $B$. Since $B$ is normal, it follows that $\alpha\in B$. Thus, $\alpha$ is an algebraic element of $B$ over $A$. As $A$ is ac in $B$, we conclude that $\alpha\in A$ and hence $A$ is normal. 
\end{proof}

We now state several lemmas establishing basic properties of pfc and qfc subalgebras. In particular, the qfc property is stable under localization, as Lemma \ref{loc} shows. 

\begin{lem} \label{fc} 
The following conditions are equivalent. 
\begin{enumerate}
\item
$A$ is fc in $B$.
\item
$A$ is pfc in $B$ and $A^*=B^*$.
\item
$A$ is qfc in $B$ and $A^*=B^*$.
\end{enumerate}
\end{lem}
\begin{proof}
Omitted.
\end{proof}

\begin{lem} \label{tower}
Let $A\subset B\subset C$ be $R$-algebras. If $A$ is qfc in $B$ and $B$ is pfc in $C$, then $A$ is qfc in $C$. 
\end{lem}
\begin{proof}
Let $f,g\in C\setminus\{0\}$ be elements satisfying $fg\in A$. Since $B$ is pfc in $C$, there exists $w\in C^*$ such that $wf, w^{-1}g\in B$. As $A$ is qfc in $B$, the condition $wf\cdot w^{-1}g=fg\in A$ implies that there are $u,v\in B^*$ such that $uwf,vw^{-1}g\in A$. Clearly $uw, vw^{-1}\in C^*$, hence $A$ is qfc in $C$. 
\end{proof}

\begin{lem} \label{strong qfc}
If for each $f\in B\setminus\{0\}$, there exists $u\in B^*$ such that $uf\in A$, then $A$ is qfc in $B$. 
\end{lem}
\begin{proof}
Obvious. 
\end{proof}

\begin{lem} \label{loc}
Let $S\subset A$ be a multiplicatively closed set. Then the following assertions hold true. 
\begin{enumerate}
\item
 $A$ is qfc in $S^{-1}A$. 
\item
If $A$ is qfc in $B$, then $S^{-1}A$ is qfc in $S^{-1}B$.
\item
Suppose that $S\subset B^*$. If $S^{-1}A$ is qfc in $S^{-1}B=B$, then $A$ is qfc in $B$. 
\end{enumerate}
\end{lem}
\begin{proof}
{\bf (1)}, {\bf (2)} Obvious.  
\medskip

{\bf (3)}  
Let $f_1,f_2\in B\setminus\{0\}$ be elements satisfying $f_1f_2\in A$. Fix $i\in\{1,2\}$. By the assumptions, $\alpha_if_i\in S^{-1}A$ for some $\alpha_i\in(S^{-1}B)^*=B^*$. Then there is $s_i\in S$ such that $s_i\alpha_if_i\in A$. Since $s_i\alpha_i\in B^*$, $A$ is qfc in $B$. 
\end{proof}

Next, we establish a localization criterion for the qfc property, which will play an important role throughout the paper.

\begin{prop} \label{new criterion} 
Let $S:=A\cap B^*$. Then $A$ is qfc in $B$ if and only if $S^{-1}A$ is qfc in $B$. Moreover, in this case, for any $A\subset C\subset S^{-1}A$, $C$ is qfc in $B$. 
\end{prop}
\begin{proof}
Since $S\subset B^*$, the assertion follows from Lemma \ref{loc}.  
\end{proof}

The following example shows that a qfc subalgebra is not necessarily generated by monomials.

\begin{example}
{\rm
Let $B:=R[x^{\pm1},y^{\pm1}]$ be the Laurent polynomial ring in two variables over $R$, $A=R[x+y,x^2,x^3]$ and $S:=A\cap B^*$. 
Since $x^2y=x^2(x+y)-x^3\in A$, we have:
\begin{align*}
S^{-1}A=R[x+y,x^2,x^3][(x^2y)^{-1},x^{-2},x^{-3}]=R[x^{\pm1},y^{\pm1}]=B. 
\end{align*} 
Therefore, by Proposition \ref{new criterion}, $A$ is qfc in $B$. 
}
\end{example}


\section{pfc/qfc subalgebras of $R^{[\pm n]}$}
Throughout this section, $R$ denotes a domain and let $B=R[x_1^{\pm1},\ldots,x_n^{\pm1}]\cong R^{[\pm n]}$ be the Laurent polynomial ring in $n$ variables over $R$. 

For $1\leq i\leq n$, let $\bm{e}_i=(0,\ldots,0,1,0,\ldots,0)$, that is, $\{\bm{e}_1,\ldots,\bm{e}_n\}$ is the standard basis of $\Z^n$. An element $f\in B$ is called a {\bf monomial} if $f\not\in R$ and there are $r\in R\setminus\{0\}$ and an $n$-tuples of integers $\bm{a}=(a_1,\ldots,a_n)\in\Z^n$ such that $f=r\bm{x}^{\bm{a}}=rx_1^{a_{1}}\cdots x_n^{a_{n}}$. If $r\in R^*$, then the monomial $f$ is said to be {\bf monic}. 

We consider the $\Z$-grading on $B$ by letting $x_1,\ldots,x_n$ be homogeneous of degree one. For each $1\leq i\leq n$, define $\deg_i:B\setminus\{0\}\to \Z$ (resp.\:$\ord_i:B\setminus\{0\}\to \Z$) by the degree (resp. order) function with respect to the variable $x_i$. That is, for $f\in B\setminus\{0\}$, $\deg_i(f)$ (resp. $\ord_i(f))$ is the greatest (resp. least) degree with respect to $x_i$. 
It is clear that $\deg_i(f)-\ord_i(f)\geq 0$ and the equality holds if and only if $f$ is a monomial as an element of $B_i[x_i^{\pm1}]\cong B_i^{[\pm1]}$, where 
\[
B_i:=R[x_1^{\pm1},\ldots,x_{i-1}^{\pm1},x_{i+1}^{\pm1},\ldots,x_n^{\pm1}].
\]


\subsection{Basic properties of qfc/pfc subalgebras}

It is well known that the coefficient ring $R$ of the polynomial ring $R^{[n]}$ is an fc subalgebra. Recall that $R$ is not fc in $B$, see Lemma \ref{fc}.  However, by replacing fc with pfc/qfc condition, a similar result holds as below.  

\begin{prop} \label{coeff}
$R$ is pfc {\rm(}hence qfc{\rm)} in $B$. 
\end{prop}
\begin{proof}
Let $f,g\in B\setminus\{0\}$ be elements satisfying $fg\in R$. For each $1\leq i \leq n$, we have
\begin{align*}
0=\deg_i(fg)-\ord_i(fg)=(\deg_i(f)-\ord_i(f))+(\deg_i(g)-\ord_i(g)). 
\end{align*}
Then $\deg_i(f)=\ord_i(f)=:d_i$ and $\deg_i(g)=\ord_i(g)=:e_i$, hence $f$ and $g$ are monomials as elements of $B_i[x_i^{\pm1}]$ for all $1\leq i\leq n$. There are $r,s\in R$ such that $f=rx_1^{d_1}\cdots x_n^{d_n}$ and $g=sx_1^{e_1}\cdots x_n^{e_n}$. Let $u=x_1^{e_1}\cdots x_n^{e_n}\in B^*$. Since $d_i+e_i=0$ for any $1\leq i\leq n$, we have $uf=r\in R$ and $u^{-1}g=s\in R$, hence $R$ is pfc in $B$. 
\end{proof}

\begin{prop} \label{pfc} 
For each $1\leq i\leq n$, let $\delta_i\in\{-1,0,1,\pm1\}$. Let $A=R[x_1^{\delta_1},\ldots,x_n^{\delta_n}]$ be the $R$-subalgebra of $B$. Then $A$ is pfc {\rm(}hence qfc{\rm)} in $B$.
\end{prop}
\begin{proof}
We use induction on $n\geq0$. If $n=0$, then $A=B=R$, thus the assertion is obvious. Assume that $n\geq1$. If there exists $i\in\{1,\ldots,n\}$ such that $\delta_i=\pm1$, then 
\[
A=R'[x_1^{\delta_1},\ldots,x_{i-1}^{\delta_{i-1}},x_{i+1}^{\delta_{i+1}},\ldots,x_n^{\delta_n}], \quad B=R'[x_1^{\pm1},\ldots,x_{i-1}^{\pm1},x_{i+1}^{\pm1},\ldots,x_n^{\pm1}], 
\]
where $R'=R[x_i^{\pm1}]$. Since $B\cong R'^{[\pm(n-1)]}$ and $A$ is an $R'$-subalgebra of $B$ generated by at most $n-1$ elements, the assertion follows from induction. Thus, we may assume that $\delta_i\in\{-1,0,1\}$ for any $1\leq i\leq n$. 

If there exists $i\in\{1,\ldots,n\}$ such that $\delta_i=-1$, then by applying the $R$-automorphism $\sigma:B\to B$ defined by $\sigma(x_i)=x_i^{-1}$ and $\sigma(x_j)=x_j$ for $j\not=i$, we have
\[
A\cong\sigma(A)=R[x_1^{\delta_1},\ldots,x_{i-1}^{\delta_{i-1}},x_{i},x_{i+1}^{\delta_{i+1}},\ldots,x_n^{\delta_n}]\subset\sigma(B)=B. 
\]
By Lemma \ref{autom} below, we may assume that $\delta_i\in\{0,1\}$ for any $1\leq i\leq n$. 

By rearranging the order of variables and by Proposition \ref{coeff}, we may further assume that there exists $1\leq m\leq n$ such that $A=R[x_1,\ldots,x_m]$. 
Let $f,g\in B\setminus\{0\}$ be elements satisfying $fg\in A$. Let $i\in\{1,\ldots,m\}$ and $j\in\{m+1,\ldots,n\}$ (if $m<n$). Then, 
\[
\ord_i(f)+\ord_i(g)=\ord_i(fg)\geq0
\]
and 
\[
\deg_j(f)=\ord_j(f),\: \deg_j(g)=\ord_j(g), \: \deg_j(f)+\deg_j(g)=0.
\] 
Here, we define 
\[
\varepsilon_i:=
	\begin{cases}
	0 \quad &(\text{if}\: \ord_i(f)=\ord_i(g)=0)\\
	\ord_i(f) \quad &(\text{if}\: \ord_i(f)<0\:\text{and}\:\ord_i(g)>0)\\
	-\ord_i(g)\quad &(\text{if}\: \ord_i(f)>0\:\text{and}\:\ord_i(g)<0)
	\end{cases}
\]
and $\varepsilon_j:=\ord_j(f)=-\ord_i(g)$. Let $u=\prod_{\ell=1}^n x_{\ell}^{-\varepsilon_{\ell}}$. 
Then, 
\begin{align*}
\ord_{i}(uf)&=-\varepsilon_{i}+\ord_{i}(f)\geq0, \quad
\ord_{i}(u^{-1}g)=\varepsilon_{i}+\ord_{i}(g)\geq0,\\
\ord_{j}(uf)&=-\varepsilon_{j}+\ord_{j}(f)=0, \quad
\ord_{j}(u^{-1}g)=\varepsilon_{j}+\ord_{j}(g)=0, \\
\deg_{j}(uf)&=-\varepsilon_{j}+\deg_{j}(f)=0, \quad
\deg_{j}(u^{-1}g)=\varepsilon_{j}+\deg_{j}(g)=0.
\end{align*}
Thus, $uf,u^{-1}g\in R[x_1,\ldots,x_m]=A. $ 
Therefore, $A$ is pfc in $B$. 
\end{proof}

\begin{lem} \label{autom}
Let $A$ be an $R$-subalgebra of $B$ and let $\sigma\in\Aut_R(B)$. Then the following assertions hold. 
\begin{enumerate}
\item There exists $\mathcal{A}\in{\rm GL}_n(\Z)$ such that $\sigma(x_i)=\bm{x}^{\bm{e}_i\mathcal{A}}$. Moreover, for $\bm{b}\in\Z^n$, $\sigma(\bm{x}^{\bm{b}})=\bm{x}^{\bm{b}\mathcal{A}}$. 
\item $A$ is pfc {\rm (}resp. qfc{\rm )} in $B$ if and only if  $\sigma(A)$ is pfc {\rm (}resp. qfc{\rm )} in $B$. 
\end{enumerate}
\end{lem}
\begin{proof}
{\bf (1)} Let $i\in\{1,\ldots,n\}$. Since $x_i\in B^*$, so is $\sigma(x_i)$ and hence it is a monic monomial in $x_1,\ldots,x_n$. Hence, there are $a_{ij}\in\Z$ such that $\sigma(x_i)=x_1^{a_{i1}}\cdots x_n^{a_{in}}=\bm{x}^{\bm{a}_i}$, where $\bm{a}_i=(a_{i1},\ldots,a_{in})$. Set $\mathcal{A}:=(^t\bm{a}_1\:\cdots\:^t\bm{a}_n)$. Since $\sigma\in\Aut_R(B)$, we have $\mathcal{A}\in{\rm GL}_n(\Z)$. 

Let $\bm{b}=(b_1,\ldots,b_n)\in\Z^n$. Then, 
\begin{align*}
\sigma(\bm{x}^{\bm{b}})=\sigma(x_1)^{b_1}\cdots\sigma(x_n)^{b_n}
=x_1^{b_1a_{11}+\cdots+b_na_{n1}}\cdots x_n^{b_1a_{1n}+\cdots+b_na_{nn}}
=\bm{x}^{\bm{b}\mathcal{A}}.
\end{align*}

{\bf (2)} Omitted. 
\end{proof}

An $R$-retract of the polynomial ring $R^{[n]}$ is not necessarily an fc (hence qfc or pfc) $R$-subalgebra (see e.g., \cite[Examples 2.1]{Cos77}). However, the following corollary shows that every $R$-retract of the Laurent polynomial ring is a pfc (hence qfc) $R$-subalgebra. 
\begin{cor} \label{ret-pfc}
Let $A$ be an $R$-subalgebra of $B$. If $A$ is an $R$-retract of $B$, then it is pfc {\rm(}hence qfc{\rm)} in $B$.
\end{cor}
\begin{proof}
The assertion follows from Proposition \ref{pfc} and Theorem \ref{ret} below. 
\end{proof}

\begin{thm} \label{ret}
Let $A$ be an $R$-retract of $B$. Then there are $r\in\N$ and $y_1,\ldots,y_n\in B$ such that $B=R[y_1^{\pm1},\ldots,y_n^{\pm1}]$ and $A=R[y_1^{\pm1},\ldots,y_r^{\pm1}]$. Therefore, $B\cong A^{[\pm(n-r)]}$. 
\end{thm}
\begin{proof}
The assertion follows from \cite[Theorem 1.2]{GN} and \cite[Proposition 3.2]{GN}. 
\end{proof}

By the following example, the converse of Corollary \ref{ret-pfc} does not hold true in general. 

\begin{example} \label{pfc-not ret}
{\rm 
It follows from Proposition \ref{pfc} that $R[x]$ is pfc in $R[x^{\pm1}]$. However, it is not an $R$-retract of $R[x^{\pm1}]$. Indeed, by Theorem \ref{ret}, an $R$-retract of $R[x^{\pm1}]$ is $R$ or $R[x^{\pm1}]$. }
\end{example}

A submonoid $M$ of $\Z^n$ is {\bf normal} if whenever $m\bm{a}\in M$ for some positive integer $m$ and $\bm{a}\in \langle M\rangle$, then $\bm{a}\in M$. It is well known that, when $M$ is finitely generated and $R$ is a field, $M$ is normal if and only if the monoid algebra $R[M]$ is normal (see e.g., \cite[p.260, Theorem 6.1.4]{BH98}). The pfc condition for a monomial algebra implies normality, as shown below. 

\begin{prop}\label{pfc normal}
If $R[M]$ is pfc in $B$, then the monoid $M$ is normal. In particular, if $M$ is finitely generated and $R$ is a field, then $R[M]$ is normal. 
\end{prop}
\begin{proof}
Let $m$ be a positive integer and $\bm{a}\in \langle M\rangle$ such that $m\bm{a}\in M$. Then $(\bm{x}^{\bm a})^m=\bm{x}^{m\bm a}\in R[M]$. Thus, 
\[
(\bm{x}^{\bm a}-1)(1+\bm{x}^{\bm a}+\cdots+(\bm{x}^{\bm a})^{m-1})=\bm{x}^{m\bm a}-1\in R[M]. 
\]
Since $R[M]$ is pfc in $B$, there exists $\bm{b}\in\Z^n$ such that 
\begin{align*}
\bm{x}^{\bm{a}+\bm{b}}-\bm{x}^{\bm b}&=\bm{x}^{\bm b}(\bm{x}^{\bm a}-1)\in R[M], \\
\bm{x}^{-\bm b}+\bm{x}^{\bm{a}-\bm{b}}+\cdots&=\bm{x}^{-\bm b}(1+\bm{x}^{\bm a}+\cdots+(\bm{x}^{\bm a})^{m-1})\in R[M]
\end{align*}
As $R[M]$ is generated by monic monomials, $\bm{x}^{\bm{a}+\bm{b}},\bm{x}^{\pm\bm b}\in R[M]$ and hence $\bm{a}+\bm{b},\pm\bm{b}\in M$. Therefore, $\bm{a}=(\bm{a}+\bm{b})+(-\bm{b})\in M$, which implies that $M$ is normal. 
\end{proof}

The converse of Proposition \ref{pfc normal} does not hold in general, see the example below. 

\begin{example} \label{normal-not pfc}
{\rm
Let $k$ be a field. For $M:=2\N$, let 
\[
A:=k[M]=k[x^2]\in k[x^{\pm1}]=:B.
\] 
Then $M$ is normal and $A$ is a normal domain since $A\cong k^{[1]}$. However, $A$ is not pfc in $B$, as we show Corollary \ref{d-th root pfc} in Section 5. In particular, $A$ is neither qfc nor ac in $B$.  
}
\end{example}


\subsection{The gap sets and qfc conditions}\label{sec32}

Let $A$ be an $R$-subalgebra of $B$. 
Define the monoid $M(A)$ of $\Z^n$ by 
\[
M(A):=\{\bm{a}=(a_1,\ldots,a_n)\in\Z^n \ | \ \bm{x}^{\bm{a}}=x_1^{a_{1}}\cdots x_n^{a_{n}}\in A\}. 
\]
For $f\in B\setminus\{0\}$, write
\[
f=\sum_{\bm{a}\in\Z^n}r_{\bm{a}}\bm{x}^{\bm{a}}, \:\: r_{\bm{a}}\in R. 
\]
Then, we define the {\bf support set} of $f$ by
\[
{\rm Supp}(f):=\{\bm{a}\in\Z^n \ | \ r_{\bm{a}}\not=0 \}, \quad {\rm Supp}(0):=\emptyset 
\]
and define the subset $S(A)$ of $\Z^n$ by 
\[
S(A):=\bigcup_{f\in A}{\rm Supp}(f).
\]
Note that if $A$ is generated by monic monomials, then $M(A)=S(A)$.

For a subset $D\subset \Z^n$, set 
\begin{align*}
\R(D)&:=\left\{\sum_{i=1}^{\ell}\alpha_i\bm{a}_i \ \middle| \ \ell\geq0, \: \bm{a}_i\in D, \: \alpha_i\in\R\right\}\subset\R^n, \\
{\rm Cone}(D)&:=\left\{\sum_{i=1}^{\ell}\alpha_i\bm{a}_i \ \middle| \ \ell\geq0, \: \bm{a}_i\in D, \: \alpha_i\in\R_{\geq0}\right\}\subset\R^n.
\end{align*}
Here, we define the set of lattice points in ${\rm Cone}(S(A))$ by 
\[
C(A):={\rm Cone}(S(A))\cap\Z^n. 
\]
Note that the following holds: $M(A) \subset S(A) \subset C(A)$. 
Define the {\bf gap set} of $A$ by 
\[
\Gap(A):=C(A)\setminus M(A). 
\]
For $\bm{u}=(u_1,\ldots,u_n),\bm{v}=(v_1,\ldots,v_n)\in\R^n$, define the inner product and the norm as below: 
\[
\langle \bm{u},\bm{v}\rangle:=u_1v_1+\cdots+u_nv_n, \quad \|\bm{u}\|:=|u_1|+\cdots+|u_n|. 
\]

In this subsection, we will prove Theorem \ref{fin-gap}. To prove this, we establish several lemmas below. 

\begin{lem} \label{inner}
Let $D$ be a finitely generated submonoid of $\Z^n$. Then the following assertions hold.
\begin{enumerate}
\item There are $s\in\N$ and $\bm{v}_1,\ldots,\bm{v}_s\in\R(D)\setminus\{0\}$ such that 
\[
{\rm Cone}(D)=\bigcap_{j=1}^s\{\bm{u}\in\R(D) \ | \ \langle \bm{u},\bm{v}_j\rangle\geq0\}. 
\]
\item There exists $\bm{w}\in D$ such that $\langle \bm{w},\bm{v}_j\rangle>0$ for all $1\leq j\leq s$. 
\end{enumerate}
\end{lem}
\begin{proof} 
{\bf (1)} Obvious (see e.g., \cite[p.30, Theorem 1.3]{Zie95}). 
\medskip

{\bf (2)} Let $r:=\dim\R(D)$ and $\bm{a}_1,\ldots,\bm{a}_r\in D$ be an $\R$-basis of $\R(D)$. Set $\bm{w}=\sum_{i=1}^r \bm{a}_i \in D$. 
Since $\bm{w},\bm{a}_i\in {\rm Cone}(D)$, it follows from (1) that $\langle \bm{w},\bm{v}_j\rangle\geq0$ and $\langle \bm{a}_i,\bm{v}_j\rangle\geq0$ for any $1\leq i\leq r$ and $1\leq j\leq s$. 
Assume that $\langle \bm{w},\bm{v}_j\rangle=0$ for some $1 \le j \le s$. Then
\[
\langle \bm{a}_1,\bm{v}_j\rangle+\cdots+\langle \bm{a}_r,\bm{v}_j\rangle=\langle \bm{w},\bm{v}_j\rangle=0, 
\] 
hence $\langle \bm{a}_i,\bm{v}_j\rangle=0$ for all $1\leq i\leq r$. Write $\bm{v}_j = \sum_{i=1}^r c_i \bm{a}_i$ for $c_i\in \mathbb{R}$. We then obtain that
\[
\langle \bm{v}_j,\bm{v}_j\rangle 
= \left\langle \sum_{i=1}^r c_i\bm{a}_i,\bm{v}_j\right\rangle
= \sum_{i=1}^r c_i \langle \bm{a}_i,\bm{v}_j\rangle = 0. 
\]
Hence $\bm{v}_j=\bm{0}$, which contradicts the choice of $\bm{v}_j$ in (1). Therefore, $\bm{w}$ is desired, that is, $\langle \bm{w},\bm{v}_j\rangle$ is positive for all $1\leq j\leq s$. 
\end{proof}

\begin{lem} \label{autom-gap}
Let $\sigma\in\Aut_RB$. For $\bm{c}\in\Z^n$, let $\bm{c}'\in\Z^n$ satisfying $\bm{x}^{\bm{c}'}=\sigma(x^{\bm{c}})$. Then $\bm{c}\in \Gap(A)$ if and only if $\bm{c}'\in \Gap(\sigma(A))$. In particular, $|\Gap(A)|=|\Gap(\sigma(A))|$.
\end{lem}
\begin{proof}
By Lemma \ref{autom} (1), for $\sigma\in\Aut_RB$, there exists $\mathcal{A}\in{\rm GL}_n(\Z)$ such that $\sigma(\bm{x}^{\bm{b}})=\bm{x}^{\bm{b}\mathcal{A}}$ for any $\bm{b}\in\Z^n$. 

Let $f\in A\setminus\{0\}$ and ${\bm{b}}\in {\rm Supp}(f)$. There are $r\in R\setminus\{0\}$ and $g\in B$ such that $f=r\bm{x}^{\bm{b}}+g$. Then, $\sigma(f)=r\bm{x}^{\bm{b}\mathcal{A}}+\sigma(g)$. 
Since $\mathcal{A}\in{\rm GL}_n(\Z)$, there are no monomials in $\sigma(g)$ of the form $\bm{x}^{\bm{b}\mathcal{A}}$, which implies that $\bm{b}\mathcal{A}\in{\rm Supp}(\sigma(f))$. Therefore, the mapping $T_{\mathcal{A}}:S(A)\ni\bm{b}\mapsto\bm{b}\mathcal{A}\in S(\sigma(A))$ is bijective. In particular, the restriction $T_{\mathcal{A}}|_{M(A)}:M(A)\to M(\sigma(A))$ is also bijective. 

Let $\bm{c}\in C(A)$. Then there are $ \ell\geq0,\bm{b}_i\in S(A)$ and $ \alpha_i\in\R_{\geq0}$ such that $\bm{c}=\sum_{i=1}^{\ell}\alpha_i\bm{b}_i$. Here, we extend $T_{\mathcal{A}}$ to $C(A)$ by
\[
T_{\mathcal{A}}(\bm{c}):=\bm{c}\mathcal{A}=\sum_{i=1}^{\ell}\alpha_i\bm{b}_i\mathcal{A}. 
\] 
Since the condition $\bm{c}\not\in M(A)$ is equivalent to the condition $\bm{c}\mathcal{A}\not\in M(\sigma(A))$, the map ${T}_{\mathcal{A}}$ gives the bijection $\Gap(A)\ni \bm{c}\mapsto \bm{c}\mathcal{A}\in\Gap(\sigma(A))$. 
\end{proof}

\begin{lem} \label{R[H]}
Let $H=\R(S(A))\cap\Z^n$ and $r=\rank_{\Z} H$. Then there are $y_1,\ldots,y_n\in B$ such that $B=R[y_1^{\pm1},\ldots,y_n^{\pm1}]$ and $R[H]=R[y_1^{\pm1},\ldots,y_r^{\pm1}]$. In particular, $B\cong R[H]^{[\pm(n-r)]}$. 
\end{lem}
\begin{proof}
For the natural inclusion $H\to G:=\Z^n$, we have the following exact sequence of $\Z$-modules:
\[
0\to H\to G\to G/H\to0. 
\]
Then $G/H$ is $\Z$-torsion free. Indeed, if $\bm{b}\in G$ is a torsion element of $G/H$, then there is $m\in\Z\setminus\{0\}$ such that $m\bm{b}\in H$. Since $\bm{b}=m^{-1}\cdot m\bm{b}\in \R(S(A))$, we have $\bm{b}\in H$. 
Thus, $H$ is a direct summand of $G$. Let $\bm{c}_1,\ldots,\bm{c}_n\in G$ be a $\Z$-basis of $G$ such that 
\[
H=\bigoplus_{i=1}^r\Z \bm{c}_i\cong\Z^r.
\]
Set $y_i=\bm{x}^{\bm{c}_i}$ for $1\leq i\leq n$. Since the $n\times n$ matrix $(\bm{c}_1\:\cdots\:\bm{c}_n)$ is an element of ${\rm GL}_n(\Z)$, we have $B=R[y_1^{\pm1},\ldots,y_n^{\pm1}]$ and $R[H]=R[y_1^{\pm1},\ldots,y_r^{\pm1}]$. 
\end{proof}

\begin{lem} \label{key lemma}
Suppose that there exists $\bm{w}\in\Z^n$ such that $\bm{w},\bm{w}+\bm{e}_i \in M(A)$ for all $1\leq i\leq n$. Then $A$ is qfc in $B$.  
\end{lem}

\begin{proof}
Set $S=A\cap B^*$. Since $\bm{w},\bm{w}+\bm{e}_i \in M(A)$, we have $\bm{x}^{\bm{w}},\bm{x}^{\bm{w}}x_i\in A\cap B^*=S$ for all $1\leq i\leq n$. Thus, 
\[
x_i=\dfrac{\:\bm{x}^{\bm{w}}x_i\:}{\bm{x}^{\bm{w}}}, \:\: x_i^{-1}=\dfrac{\:\bm{x}^{\bm{w}}\:}{\:\bm{x}^{\bm{w}}x_i\:}\in S^{-1}A. 
\]
Therefore $S^{-1}A=B$. It follows from Proposition \ref{new criterion} that $A$ is qfc in $B$. 
\end{proof}


Here, we prove Theorem \ref{fin-gap}. 

\begin{proof}[Proof of Theorem \ref{fin-gap}] 
Suppose that $\Gap(A)$ is finite. Set $H:=\R(S(A))\cap\Z^n$ and $r:=\rank_{\Z} H$. Then $A$ is an $R$-subalgebra of the monoid algebra $R[H]$. Let $\{\bm{a}_1,\ldots,\bm{a}_r\}\subset S(A)$ be an $\R$-basis of $\R(S(A))$ and set $D:=\sum_{i=1}^r\N\bm{a}_i$. We note that the following inclusions are hold: 
\[
D\subset S(A)\subset C(A) \subset H. 
\]
By Lemmas \ref{autom} and \ref{R[H]}, there exists $\mathcal{C}=(c_{ij})\in{\rm GL}_n(\Z)$ such that $B=R[y_1^{\pm1},\ldots,y_n^{\pm1}]$ and $R[H]=R[y_1^{\pm1},\ldots,y_r^{\pm1}]$, where $y_i=\bm{x}^{\bm{c}_i}$ and $\bm{c}_i=(c_{i1},\ldots,c_{in})$ for $1\leq i\leq n$. Hence, Proposition \ref{coeff} implies that $R[H]$ is pfc in $B$. It follows from Lemma \ref{tower} that it is enough to prove that $A$ is qfc in $R[H]$.

Since $D$ is finitely generated, it follows from Lemma \ref{inner} that there are $s\in\N$ and $\bm{v}_1,\ldots,\bm{v}_s\in\R(D)$ such that 
\begin{equation} \label{hyperplane}
{\rm Cone}(D)=\bigcap_{j=1}^s\{\bm{u}\in\R(D) \ | \ \langle \bm{u},\bm{v}_j\rangle\geq0\} 
\end{equation}
and there exists $\bm{w}\in D$ such that $ \langle \bm{w},\bm{v}_j\rangle>0$ for each $j\in \{1,\ldots,s\}$. Let $N$ be the positive integer satisfying $N>\max\{N_1,N_2\}$, where 
\begin{align*}
N_1&:=\max\left\{\left.\dfrac{|\langle\bm{c}_i,\bm{v}_j\rangle|}{\langle\bm{w},\bm{v}_j\rangle} \ \right| \  1\leq i\leq r, \:1\leq j\leq s \right\}, \\
N_2&:=\max\{\|\bm{g}\|+\|\bm{c}_i\| \ | \ \bm{g}\in{\rm Gap}(A), \: 1\leq i\leq r\}. 
\end{align*}
Then the following claim holds. 
\medskip

\noindent
{\bf Claim.}\: 
$N\bm{w},N\bm{w}+\bm{c}_i\in M(A)$ for $1\leq i \leq r$. 
\medskip 

\noindent 
{\it Proof of Claim.}
First, we show that $N\bm{w},N\bm{w}+\bm{c}_i\in C(A)$ for $1\leq i \leq r$. Since $D\subset C(A)$, $N\bm{w}\in C(A)$. For $1\leq i\leq r$ and $1\leq j\leq s$, 
\[
\langle N\bm{w}+\bm{c}_i,\bm{v}_j\rangle
=N\langle \bm{w},\bm{v}_j\rangle+\langle \bm{c}_i,\bm{v}_j\rangle
> N_1\langle \bm{w},\bm{v}_j\rangle+\langle \bm{c}_i,\bm{v}_j\rangle
\geq0. 
\]
By the equation (\ref{hyperplane}), $N\bm{w}+\bm{c}_i\in {\rm Cone}(D)$ and hence
\[
N\bm{w}+\bm{c}_i\in \left({\rm Cone}(D)\cap\Z^n\right)\subset \left({\rm Cone}(S(A))\cap\Z^n\right)=C(A). 
\]

In order to show the claim, take $\bm{g}\in{\rm Gap}(A)$. Then
\[
\|N\bm{w}\|=N\|\bm{w}\| \geq N>N_2>\|\bm{g}\|
\]
and
\begin{align*}
\|N\bm{w}+\bm{c}_i\|
>N_2\|\bm{w}\|-\|\bm{c}_i\|
>N_2-\|\bm{c}_i\|
\geq \|\bm{g}\|.
\end{align*}
Therefore, for $1\leq i \leq r$, we have $N\bm{w},N\bm{w}+\bm{c}_i\in C(A)\setminus \rm{Gap}(A)=M(A)$. This proves the claim.
\medskip

Let $\sigma\in\Aut_R(B)$ be the automorphism defined by $\sigma(y_i)=x_i$ for $1\leq i\leq n$.  
Then
\begin{align*}
\sigma(\bm{x}^{N\bm{w}+\bm{c}_i})=\sigma(\bm{x}^{N\bm{w}})\cdot\sigma(y_i)=\sigma(\bm{x}^{N\bm{w}})\cdot x_i. 
\end{align*}
Set $\bm{x}^{\bm{w}_0}:=\sigma(\bm{x}^{N\bm{w}})$ for some $\bm{w}_0\in\Z^n$. Then $\sigma(\bm{x}^{N\bm{w}+\bm{c}_i})=\bm{x}^{\bm{w}_0+\bm{e}_i}$. 
It follows from Lemma \ref{autom-gap} and Claim that $\bm{w}_0, \bm{w}_0+\bm{e}_i\in M(\sigma(A))$ for $1\leq i\leq r$. By applying Lemma \ref{key lemma} to $\bm{w}_0$, we see that  $\sigma(A)$ is qfc in $\sigma(R[H])$. Thus, by Lemma \ref{autom} (2), $A$ is qfc in $R[H]$. 

The proof of Theorem \ref{fin-gap} is complete.
\end{proof}


\section{Subalgebras generated by monic monomials}

In this section, we use the terminology introduced in Section 3 and consider a subalgebra $A$ of $B=R[x_1^{\pm1},\ldots,x_n^{\pm1}]$ which is generated by monic monomials. Let $S:=A\cap B^*$. 

The following is the main result in this section. 

\begin{thm} \label{main}
Let $A$ be an $R$-subalgebra generated by monic monomials over $R$.  Then the following conditions are equivalent.
\begin{enumerate}
\item $A$ is qfc in $B$.
\item $M(B)/M(S^{-1}A)$ is a torsion free $\Z$-module. 
\item There exists $r\in\N$ and $y_1,\ldots,y_n\in B$ such that $B=R[y_1^{\pm1},\ldots,y_n^{\pm1}]$ and $S^{-1}A=R[y_1^{\pm1},\ldots,y_r^{\pm1}]$. Therefore, $B\cong (S^{-1}A)^{[\pm(n-r)]}$. 
\item $\Gap(S^{-1}A)=\emptyset$. 
\end{enumerate}
\end{thm}
\begin{proof} 
Since $A$ is generated by monic monomials, $A=R[M(A)]$.  
Set $G=M(B)$ and $H=M(S^{-1}A)$. Note that, $H$ is the group generated by $M(A)$ and the monoid algebras are $R[G]=B$ and $R[H]=S^{-1}A$ respectively. 

\medskip
{\bf (1)} $ \Longrightarrow$  {\bf (2).}  
Suppose that $A$ is qfc in $B$. By Proposition \ref{new criterion}, $S^{-1}A$ is also qfc in $B$.
Let $\bm{t}=(t_1,\ldots,t_n)\in G$ be a torsion element of $G/H$. Then there exists $\alpha\geq1$ such that $\alpha\bm{t}\in H$. There are $r\geq1$, $c_1,\ldots,c_r\in\Z$ and $\bm{a}_1,\ldots,\bm{a}_r\in M(A)$ such that $\alpha\bm{t}=\sum_{i=1}^rc_i\bm{a}_{i}$. Then $(\bm{x}^{\bm{t}})^{\alpha}=\prod_{i=1}^r(\bm{x}^{\bm{a}_{i}})^{c_i}\in S^{-1}A$, hence
\[
(\bm{x}^{\bm{t}}-1)\cdot\sum_{j=0}^{m-1}(\bm{x}^{\bm{t}})^j=(\bm{x}^{\bm{t}})^{\alpha}-1\in S^{-1}A.
\]
As $S^{-1}A$ is qfc in $B$, there exists $\bm{b}=(b_1,\ldots,b_n)\in G$ such that 
\[
\bm{x}^{\bm{b+t}}-\bm{x}^{\bm{b}}=\bm{x}^{\bm{b}}(\bm{x}^{\bm{t}}-1)\in S^{-1}A.
\] 
Since $S^{-1}A$ is generated by monic monomials, we have $\bm{x}^{\bm{b+t}},\bm{x}^{\bm{b}}\in S^{-1}A$ and hence $\bm{b}+\bm{t}, \bm{b}\in H$. Thus, $\bm{t}=(\bm{b}+\bm{t})-\bm{b}\in H$, which implies that $G/H$ is torsion free. 

\medskip
{\bf (2)} $ \Longrightarrow$  {\bf (3).}  
Suppose that $G/H$ is torsion free. Since $G\cong\Z^n$ is a free module over a PID $\Z$, $H$ is also free and a direct summand of $G$. 
Let $r\in\{0,\ldots,n\}$ be the $\Z$-rank of $H$ and $\bm{b}_1,\ldots,\bm{b}_n\in G$ be a $\Z$-basis of $G$ with $H=\bigoplus_{i=1}^r\Z\bm{b}_i$. Set $y_i=\bm{x}^{\bm{b}_i}$ for $1\leq i\leq r$. Then $B=R[G]=R[y_1^{\pm1},\ldots,y_n^{\pm1}]$ and $S^{-1}A=R[H]=R[y_1^{\pm1},\ldots,y_r^{\pm1}]$. 

\medskip
{\bf (3)} $ \Longrightarrow$  {\bf (4).}  
Suppose that there are $r\in\N$ and $y_1,\ldots,y_n\in B$ such that $B=R[y_1^{\pm1},\ldots,y_n^{\pm1}]$ and $S^{-1}A=R[y_1^{\pm1},\ldots,y_r^{\pm1}]$. Let $\sigma\in\Aut(B)$ be the $R$-algebra automorphism of $B$ defined by $\sigma(y_i)=x_i$ for $1\leq i\leq n$. Then $\sigma(S^{-1}A)=R[x_1^{\pm1},\ldots,x_r^{\pm1}]$ and hence $\Gap(\sigma(S^{-1}A))=\emptyset$. It follows from Lemma \ref{autom-gap} that $\Gap(S^{-1}A)$ is also the empty set. 

\medskip
{\bf (4)} $ \Longrightarrow$  {\bf (1).}  
By Theorem \ref{fin-gap}, $S^{-1}A$ is qfc in $B$. Therefore, it follows from Proposition \ref{new criterion} that $A$ is qfc in $B$. 
\end{proof}

The following example shows that the condition (4) of Theorem \ref{main} does not imply the finiteness of $\Gap(A)$. In particular, the converse of Theorem \ref{fin-gap} does not hold in general. 

\begin{example} \label{infinite}
{\rm
Let $B:=R[x^{\pm1},y^{\pm1}]$ be the Laurent polynomial ring in two variables over $R$, $A:=R[x^2,y^2,x^2y^3,x^3y^2,x^3y^3]$ and $S=A\cap B^*$. Then $S^{-1}A=B$ and hence $\Gap({S^{-1}A})=\emptyset$. By Theorem \ref{main}, $A$ is qfc in $B$. However, since
\begin{align*}
M(A)=\{(2m,0),(0,2m),(a,b) \ | \ m\geq0,\:a\geq2,\: b\geq2\}, \:\: 
C(A)=\{(a,b) \ | \ a,b\in\N\}, 
\end{align*}
we have
\[
\Gap(A)=\{(2m+1,0),(0,2m+1),\:(m,1),(1,m)\ | \ m\in\Z_{>0}\}.  
\]
This implies that $\Gap(A)$ is infinite. 
}
\end{example}

\subsection{Proof of Theorem \ref{higher numerical semigroup}}
Here, we prove Theorem \ref{higher numerical semigroup}. Let $M$ be a submonoid of $\Z^n$, $B=R[\Z^n]\cong R^{[\pm n]}$ and $S=R[M]\cap B^*$. 

\begin{proof}[Proof of Theorem \ref{higher numerical semigroup}]
As $R[\langle M\rangle]=S^{-1}R[M]$, the equivalence of (1) and (2) follows from Proposition \ref{new criterion}. 
Since 
\[
\langle M\rangle=M(S^{-1}R[M]) \quad \text{and} \quad {\rm Cone}(\langle M\rangle)\cap\Z^n=C(S^{-1}R[M]),
\]
we have
\[
{\rm Cone}(\langle M\rangle)\cap\Z^n\setminus \langle M\rangle={\rm Gap}(S^{-1}R[M]). 
\]
Hence (3) and (4) are equivalent to conditions (2) and (4), respectively, of Theorem \ref{main}.

Therefore, the proof of Theorem \ref{higher numerical semigroup} is complete.
\end{proof}

\begin{cor} \label{ac}
If $R[M]$ is ac in $B$, then it is qfc in $B$. 
\end{cor} 
\begin{proof}
Set $A:=R[M]$ and $S=A\cap B^*$. Then $M=M(A)$ and $\langle M\rangle=M(S^{-1}A)$. If $m\bm{a}\in \langle M\rangle$ for some $\bm{a}\in M(B)=\Z^n$ and $m\in\N$, then $(\bm{x}^{\bm{a}})^m=\bm{x}^{m\bm{a}}\in S^{-1}A$. Thus, $\eta:=\bm{x}^{\bm{a}}$ is integral over $S^{-1}A$. There are $r\geq0$, $a_1,\ldots,a_r\in A$ and $s_1,\ldots,s_r\in S$ such that 
\[
\eta^r+\dfrac{\:a_1\:}{s_1}\eta^{r-1}+\cdots+\dfrac{\:a_{r-1}\:}{s_{r-1}}\eta+\dfrac{\:a_r\:}{s_r}=0. 
\]
By letting $s=s_1\cdots s_r$ and $\widetilde{a}_i=sa_is_i^{-1}\in A$ for $1\leq i\leq r$, we have
\[
s\eta^r+\widetilde{a}_1\eta^{r-1}+\cdots+\widetilde{a}_{r-1}\eta+\widetilde{a}_r=0, 
\]
which implies that $\eta$ is algebraic over $A$. By the assumption on $A$, $\eta\in A$ and hence $\bm{a}\in M\subset \langle M\rangle$. Therefore, $M(B)/\langle M\rangle$ is $\Z$-torsion free. It follows from Theorem \ref{main} that $A$ is qfc in $B$. 
\end{proof}

However, the converse of Corollary \ref{ac} does not hold in general. Such example is in Example \ref{non-normal}.

\subsection{One variable case}
When $B$ is the Laurent polynomial ring in one variable over $R$, the conditions of Theorem \ref{main} can be expressed in a simpler form as below. In particular, the converse of Theorem \ref{fin-gap} holds. 

\begin{thm} \label{one var}
Let $B=R[x^{\pm1}]$ be the Laurent polynomial ring in one variable over $R$ and $A$ be an $R$-subalgebra of $B$ generated by monic monomials. Assume that $A\not=R$. Set $S=A\cap B^*$. Then the following conditions are equivalent.
\begin{enumerate}
\item $A$ is qfc in $B$.
\item $\gcd(a \ | \ a\in M(A))=1$. 
\item $S^{-1}A=B$. 
\item $\Gap(A)$ is a finite set. 
\end{enumerate}
\end{thm}
\begin{proof}
Let $d=\gcd(a \ | \ a\in M(A))$. Then 
\[
M(S^{-1}A)\cong\sum_{a\in M(A)}\Z a=d\Z.
\] 
It follows that $d=1$ if and only if $M(B)/M(S^{-1}A)$ is $\Z$-torsion free. Therefore, the equivalence of (1), (2) and (3) follows from Theorem \ref{main}. 

\medskip 
{\bf (2)} $ \Longrightarrow$  {\bf (4).} 
Suppose that (2) holds. If $A\subset R[x]$, then $M(A)\subset \N$. Since $d=1$, $M(A)$ is a numerical semigroup. Therefore, $\Gap(A)=\N\setminus M(A)$ must be finite (see e.g., \cite[Lemma 2.1]{RG}). In the case where $A\subset R[x^{-1}]$, similar argument implies that $\Gap(A)$ is finite. Now, we suppose that $A\not\subset R[x]$ and $A\not\subset R[x^{-1}]$. Since the monoid $M(A)$ contains positive and negative integers, it follows from Lemma \ref{submonoid} below that $M(A)$ should be a group and hence $M(A)=d\Z=\Z=C(A)$. Therefore, $\Gap(A)=\emptyset$. 

\end{proof}

\begin{lem} \label{submonoid}
Let $M$ be a submonoid of $\Z$. If $M$ contains positive and negative integers, then $M=d\Z$ for some $d\in\Z$. 
\end{lem}
\begin{proof}
Let $d:=\min\{ a\in M \ | \ a>0\}$. For a positive integer $a\in\Z_{>0}$ with $-a\in M$, there exists $\ell\in\Z_{>0}$ such that $-a+\ell d\in\{0,1,\ldots,d-1\}$. As $-a+\ell d\in M$, the minimality of $d$ implies that $a=\ell d$. Thus,
\[
-d=-\ell d+(\ell -1)d=-a+(\ell -1)d\in M, 
\]
which implies that $d\Z\subset M$. 
For any $b\in M$, there exists $q\in\Z$ such that $b-qd\in\{0,1,\ldots,d-1\}$. Since $b-q d\in M$, we have $b=qd\in d\Z$, hence $M=d\Z$.
\end{proof}

By reinterpreting Theorem \ref{one var} in terms of monoid algebras, we have the following result. 

\begin{cor} \label{numerical semigroup}
Let $M$ be a submonoid of $\N$. Then $R[M]$ is qfc in $B=R[x^{\pm1}]$ if and only if $M$ is numerical semigroup. 
\end{cor}


\section{pfc/qfc subalgebras of $R^{[\pm1]}$}

In this section, we consider general subalgebras that are not necessarily generated by monic monomials in one variable. 
Let $A\subset B:=R[x^{\pm1}]$ such that $M(A)\not=\{0\}$. Let $S:=A\cap B^*$ and $d:=\gcd(a \ | \ a\in M(A))$. Then $M(S^{-1}A)=d\Z$. 

The following theorem generalizes Theorem \ref{one var} to the case where $A$ is not necessarily generated by monic monomials. 

\begin{thm} \label{d-th root} Let $A$, $S$, $B$ and $d$ be as above. 
Suppose that $R$ contains a field $k$ such that $|k|\geq d$. Then the following conditions are equivalent. 
\begin{enumerate}
\item $A$ is qfc in $B$. 
\item $d=1$. 
\item $S^{-1}A=B$. 
\item $\Gap(A)$ is a finite set. 
\end{enumerate}
\end{thm}
\begin{proof}
The implications {(3)} $ \Longrightarrow$ {(1)} and  {(4)} $ \Longrightarrow$ {(1)} follow from Proposition \ref{new criterion} and Theorem \ref{fin-gap} respectively. 

\medskip
{\bf (1)} $ \Longrightarrow$  {\bf (2).}  
Suppose that $A$ is qfc in $B$. By Proposition \ref{new criterion}, $S^{-1}A$ is also qfc in $B$. To derive a contradiction, we assume that $d\geq2$. Since $M(S^{-1}A)=d\Z$, every monic monomial belonging to $S^{-1}A$ is of the form $x^{md}$ for some $m\in\Z$. 

Let $\lambda_1,\lambda_2,\ldots,\lambda_{d-1}\in k\setminus\{0\}$ be distinct non-zero elements of $k$. Let $i\in\{1,2,\ldots,d-1\}$. Then
\[
(x-\lambda_i)(x^{d-1}+\lambda_ix^{d-2}+\cdots+\lambda_i^{d-2}x+\lambda_i^{d-1})=x^{d_i}-\lambda_i^{d_i}\in S^{-1}A. 
\]
As $S^{-1}A$ is qfc in $B$, there exists $a_i\in\Z$ such that $x^{a_i}(x-\lambda_i)\in S^{-1}A$. Since $R[x^{\pm d}]\subset S^{-1}A$, by reducing each $a_i$ modulo $d$, we may assume that $0\leq a_i\leq d-1$. 
If $a_i=0$, then $x=(x-\lambda_i)+\lambda_i\in S^{-1}A$, which is a contradiction. If $a_i=d-1$, then $x^{d-1}=\lambda_i^{-1}(x^d-x^{d-1}(x-\lambda_i))\in S^{-1}A$, which is also a contradiction. 
Hence, $a_i\in\{1,2,\ldots,d-2\}$, that is, there are only $d-2$ possibilities. Therefore, there exist $1\leq i<j\leq d$ such that $a_i=a_j$. Then
\[
x^{a_i}=\dfrac{1}{\lambda_j-\lambda_i}(x^{a_i}(x-\lambda_i)-x^{a_j}(x-\lambda_j))\in S^{-1}A,
\]
which is a contradiction. 

Therefore, we have that $d=1$. 

\medskip 
{\bf (2)} $ \Longrightarrow$  {\bf (3).} Since $M(S^{-1}A)=d\Z$, we have
\[
R[x^{\pm d}]=R[M(S^{-1}A)]\subset S^{-1}A\subset B. 
\]
Therefore, if $d=1$, then $S^{-1}A=B$. 

\medskip 
{\bf (2)} $ \Longrightarrow$  {\bf (4).} The assertion follows from the same arguments as in the proof of Theorem \ref{one var}, the part {(2)} $ \Longrightarrow$  {(4)}. 
\end{proof}

\begin{cor} \label{d-th root pfc}
Suppose that $R$ contains a field $k$ such that $|k|\geq d$. Then the following conditions are equivalent. 
\begin{enumerate}
\item $A$ is pfc in $B$. 
\item $A=R[x^{\delta}]$, where $\delta\in\{-1,1,\pm1\}$. 
\end{enumerate}
\end{cor}
\begin{proof}
The implication {(2)} $ \Longrightarrow$  {(1)} follows from Proposition \ref{pfc}. 

\medskip
{\bf (1)} $ \Longrightarrow$  {\bf (2).}  
Suppose that $A$ is pfc in $B$. Then it is also qfc in $B$. By Theorem \ref{d-th root}, $\Gap(A)$ is a finite set and $M(S^{-1}A)=\Z$. 
If $A\subset R[x]$, then there exists $a\in \N$ such that $a\in M(A)$, that is, $x^a\in A$. Then, 
\[
(x-1)(x^{a-1}+\cdots+x+1)=x^a-1\in A. 
\]
The pfc condition implies that there exists $m\in\Z$ such that \[
x^m(x-1), \: \: x^{-m}(x^{a-1}+\cdots+x+1)\in A.
\]
Since $A\subset R[x]$, $m$ must be $0$, hence $x\in A$. Therefore, $A=R[x]$. 
If $A\subset R[x^{-1}]$, then there exists $a\in \N$ such that $x^{-a}\in A$. Thus, a similar argument shows that $A=R[x^{-1}]$. 

Suppose that $A\not\subset R[x]$ and $A\not\subset R[x^{-1}]$. Since $C(A)=\Z$ and $\Gap(A)$ is finite, there are $a,b\in\N$ such that $a,-b\in M(A)$. By Lemma \ref{submonoid}, $M(A)$ is a group, hence $M(A)=M(S^{-1}A)=\Z$. Therefore, $x,x^{-1}\in A$ and hence $A=R[x^{\pm1}]$. 
\end{proof}

In Theorem \ref{d-th root}, we assume that $M(A)\not=\{0\}$. On the other hand, there is an example of a qfc subalgebra $A$ such that $M(A)=\{0\}$ as below. 

\begin{ex}\label{niceex}
Let $k=\mathbb{Z}/2\mathbb{Z}$, $B:=k[x^{\pm1}]$ and  
\[
A:=k[f\in k[x] \ | \ \text{$f$ is irreducible of degree $\geq 2$}].
\]
Then the following assertions hold true. 
\begin{enumerate} 
\item There are no monomials in $A$, hence $M(A)=\{0\}$.  
\item $\Gap(A)$ is infinite. 
\item $A$ is qfc in $B$.
\end{enumerate}
\end{ex}
\begin{proof}
Let $f\in k[x]$ be an irreducible polynomial of degree $\geq2$. Then $f(0)=f(1)=1$. 

\medskip
{\bf (1)} Assume that there exists $n\geq1$ such that $x^n\in A$. Then there are $m\geq1$ and irreducible polynomials $f_1,\ldots,f_m\in k[x]$ of degree at least two such that
\[
x^n=\sum_{\bm{a}\in\N^m}\delta_{\bm{a}}f_1^{a_1}\cdots f_m^{a_m}, 
\]
where $\bm{a}=(a_1,\ldots,a_m)\in\N^m$, $\delta_{\bm{a}}\in k$ and $I:=\{\bm{a} \ | \ \delta_{\bm{a}}\not=0\}$ is a finite set. Let $\epsilon_{\delta}:k[x]\to k$ be the evaluation map defined by $\epsilon_{\delta}(x)=\delta$, where $\delta\in k=\{0,1\}$. Then
\[
0=\epsilon_{0}(x^n)=\sum_{\bm{a}\in I}\delta_{\bm{a}}f_1(0)^{a_1}\cdots f_m(0)^{a_m}=\sum_{\bm{a}\in I}1 
\]
and 
\[
1=\epsilon_{1}(x^n)=\sum_{\bm{a}\in I}\delta_{\bm{a}}f_1(1)^{a_1}\cdots f_m(1)^{a_m}=\sum_{\bm{a}\in I}1. 
\]
This is a contradiction. Therefore, $A$ does not contain monomials and hence $M(A)=\{0\}$. 
\medskip

{\bf (2)} Since $C(A)=\N$, $\Gap(A)=\N\setminus\{0\}$ is infinite. 
\medskip

{\bf (3)} Let $f\in B\setminus\{0\}$. Then there exists $n\geq0$ such that $x^nf\in k[x]$. Taking the irreducible decomposition of $x^nf$ in $k[x]$, we obtain
\[
x^nf=x^{a}(x+1)^bg_1^{c_1}\cdots g_m^{c_m}, 
\]
where $m\geq0$, $a,b,c_i\in \N$ and each $g_i\in k[x]$ is irreducible of degree at least two. Since $x^2+x+1$ is irreducible in $k[x]$, 
\[
x(x+1)=x^2+x=(x^2+x+1)+1\in A. 
\]
Therefore, 
\[
x^{n-a+b}f=(x(x+1))^bg_1^{c_1}\cdots g_m^{c_m}\in A. 
\]
By Lemma \ref{strong qfc}, $A$ is qfc in $B$. 
\end{proof}

\section{A Question of Chakraborty, Gurjar and Miyanishi}

In \cite{CGM15}, Chakraborty, Gurjar and Miyanishi gave the following question. 

\begin{question} \label{CGM}
{\rm 
Is any fc subalgebra of a PID also a PID?
}
\end{question}
In 2019, Chakraborty, Gurjar and Mondal \cite{CGM19} gave an example which showed that some conditions were necessary for the above question to be true. More precisely, they proved that 
any countable UFD can be embedded in a countable PID as an fc subalgebra (see \cite[Theorem 2.2]{CGM19}). 
However, the following more general theorem had already been established by Cohn in 1968 (see \cite[Theorem 3.4 and Corollary 1]{Coh68}).

\begin{thm}
For any UFD $A$, there exists a PID $B$ such that $A$ is fc in $B$. 
\end{thm}

For the reader's convenience, we describe the construction of $B$. Recall that a domain $R$ is called an {\bf HCF} ring if, for any $a,b\in R$, $aR\cap bR=cR$ for some $c\in R$. $R$ is said to be {\bf atomic} if every non-zero non-unit of $R$ can be written as a finite product of irreducible elements of $R$.
In particular, an atomic HCF ring is a UFD.

Let $A$ be a UFD and $A[x,y]$ be the polynomial ring in two variables over $A$. Here, we define
\[
\mathcal{P}_A:=\{(a,b)\in A^2 \ | \ \gcd(a,b)=1\} 
\]
and 
\[
\mathfrak{S}_A:=\left\{\left.\prod(ax+by) \ \right| \ (a,b)\in \mathcal{P}_A\right\}, 
\]
where we consider products consisting of finitely many elements, allowing repetitions and the empty product. Then $\mathfrak{S}_A$  is a multiplicatively closed set of $A[x,y]$. For $z,w\in\mathfrak{S}_A$, we define $z<w$ if $w=\xi z$ for some $\xi\in \mathfrak{S}_A$. Then $(\mathfrak{S}_A,<)$ is a direct set. For $z\in \mathfrak{S}_A$, $A[x,y,z^{-1}]$ is a graded domain with the standard grading. Here, we define
\[
A\langle z^{-1}\rangle:=A[x,y,z^{-1}]_0. 
\]
Note that $A\langle 1^{-1}\rangle=A$. For $w\in \mathfrak{S}_A$ satisfying $z<w$, we define the ring homomorphism $\varphi_{wz}:A\langle z^{-1}\rangle\to A\langle w^{-1}\rangle$ by 
\[
\varphi_{wz}\left(\dfrac{f}{\:z^m\:}\right)=\dfrac{\:{\xi}^mf\:}{{\xi}^mz^m}=\dfrac{\:{\xi}^mf\:}{w^m}
\]
where $m\in\N$ and $f\in A[x,y]$ is a homogeneous polynomial of degree $m\deg z$. Then we see that  $(A\langle z^{-1}\rangle,\varphi_{wz})$ is a direct system. Let $\A_1:=\varinjlim A\langle z^{-1}\rangle$ be the direct limit of $(A\langle z^{-1}\rangle,\varphi_{wz})$. 

\begin{lem} The following assertions hold true. 
\begin{enumerate}
\item For each $z\in\mathfrak{S}_A$, $A\langle z^{-1}\rangle$ is a UFD.
\item Let $z,w\in \mathfrak{S}_A$ such that $z<w$. By viewing $A\langle z^{-1}\rangle$ as a subalgebra of $A\langle w^{-1}\rangle$ via the injective map $\varphi_{wz}:A\langle z^{-1}\rangle\to A\langle w^{-1}\rangle$, $A\langle z^{-1}\rangle$ is fc in $A\langle w^{-1}\rangle$. In particular, if we take $z=1$, it follows that $A$ is fc in $A\langle w^{-1}\rangle$. 
\item $A$ is fc in $\A_1$ and $\A_1$ is a UFD. 
\item Every pair $(a,b)\in \mathcal{P}_A$ is comaximal in  $\A_1$, that is, $1\in a\A_1+b\A_1$. 
\end{enumerate}
\end{lem}
\begin{proof}
{\bf (1)} For $z\in \mathfrak{S}_A$, let $R:=A\langle z^{-1}\rangle$. Then there are $m\geq 0$, $(a_i,b_i)\in \mathcal{P}_A$ and $n_i\geq1$ for $1\leq i\leq m$, such that
\[
z=\prod_{i=1}^m(a_ix+b_iy)^{n_i}, 
\]
hence 
\[
R=A{\left[x,y,\dfrac{1}{\:a_1x+b_1y\:},\ldots,\dfrac{1}{\:a_mx+b_my\:}\right]_0}.
\]
To show that $R$ is a UFD, it suffices to prove that $R$ is an atomic HCF ring. 

Let $\alpha=fz^{-s},\beta=gz^{-t}\in R\setminus\{0\}$, where $s,t\geq0$ and $f,g\in A[x,y]$ are homogeneous of degree $s\deg z$, $t\deg z$ respectively. We may assume that $\gcd(f,z)=\gcd(g,z)=1$. Let $h:=\gcd(f,g)\in A[x,y]$ and $L$ be a product of ($\deg h$)-elements of $\{a_1x+b_1x,\ldots,a_mx+b_m\}$. Then $\alpha\beta h^{-1}L\in R$ and hence 
\[
\alpha R\cap\beta R=\dfrac{\:\alpha\beta \:}{\:hL^{-1}\:}R. 
\]
Thus, $R$ is an HCF ring. 

Suppose that $\alpha$ is not a unit. Taking the prime decomposition of $f$ by $f=f_1\cdots f_r$, where each $f_i$ is homogeneous and prime in $A[x,y]$ of degree $d_i$. Then
\begin{equation} \label{irred decomp}
\alpha=\dfrac{\:f\:}{\:z^s\:}=\dfrac{\:f_1\:}{\:L_1\:}\cdots \dfrac{\:f_r\:}{\:L_r\:},
\end{equation}
where $L_i$ is a product of $d_i$-elements of $\{a_1x+b_1x,\ldots,a_mx+b_m\}$ satisfying $\gcd(f_i,L_i)=1$ and $z^s=L_1\cdots L_r$. Then $f_iL_i^{-1} $is irreducible in $R$  for each $1\leq i\leq r$. Indeed, if there are $s_j\geq0$ and homogeneous elements $g_j\in A[z,y]$ of degree $s_j\deg z$ with $\gcd(g_j,z)=1$ for $j\in\{1,2\}$, such that 
\[
\dfrac{\:f_i\:}{\:L_i\:}=\dfrac{\:g_1\:}{\:z^{s_1}\:}\dfrac{\:g_2\:}{\:z^{s_2}\:}. 
\]
Since $g_j$ and $z$ are relatively prime for $j\in\{1,2\}$, we have $z^{s_1+s_2}L_i^{-1}=1$ and hence $f_i=g_1g_2$ in $A[x,y]$. By the irreducibility of $f_i$, $g_1=1$ or $g_2=1$. Thus, $f_iL_i^{-1} $is irreducible in $R$. 

Therefore, the equation (\ref{irred decomp}) gives an irreducible decomposition of $\alpha$ in $R$. Hence, $R$ is atomic.  

\medskip
{\bf (2)} Suppose that $w=\xi z$ for some $\xi\in \mathfrak{S}_A$. Let $\alpha,\beta\in A\langle w^{-1}\rangle\setminus\{0\}$ satisfying $\alpha\beta\in A\langle z^{-1}\rangle$. Since 
\[
\alpha,\beta\in A\langle w^{-1}\rangle\subset A[x,y,w^{-1}]=A[x,y,z^{-1}][\xi^{-1}], 
\]
there are $f,g\in A[x,y,z^{-1}]$ and $s,t\geq0$ such that $\alpha={f}\xi^{-s}$ and $\beta={g}\xi^{-t}$, where $f,g$ are homogeneous of degrees $s\deg \xi$ and $t\deg \xi$ respectively. As $A[x,y,z^{-1}]$ is a UFD, we may assume that that $\gcd(f,\xi)=\gcd(g,\xi)=1$. 
Then the following equation holds in $A[x,y,z^{-1}]$: 
\[
\xi^{s+t}\alpha\beta=fg.
\] 
Since neither $f$ nor $g$ is divisible by $\xi$ in $A[x,y,z^{-1}]$, we have $s=t=0$. Therefore, $\alpha=f$ and $\beta=g$ belong to $A[x,y,z^{-1}]_0=A\langle z^{-1}\rangle$, which implies that $A\langle z^{-1}\rangle$ is fc in $A\langle w^{-1}\rangle$. 

\medskip
{\bf (3)} Let $\alpha,\beta\in \A_1\setminus\{0\}$ satisfying $\alpha\beta\in A$. We can choose $z\in\mathfrak{S}_A$ so that $\alpha,\beta\in A\langle z^{-1}\rangle$. By (2), since $A$ is fc in $A\langle z^{-1}\rangle$, we have $\alpha\in A$ and $\beta \in A$, which implies that $A$ is fc in $\A_1$. 

To show that $\A_1$ is a UFD, it suffices to prove that for each $z\in\mathfrak{S}_A$, every prime element of $A\langle z^{-1}\rangle$ remains prime in $\A_1$. Let $p\in A\langle z^{-1}\rangle$ be a prime element in $A\langle z^{-1}\rangle$. Assume that $\alpha\beta\in p\A_1$ for some $\alpha,\beta\in \A_1$. We can choose $w\in\mathfrak{S}_A$ with $z<w$ so that $p,\alpha,\beta\in A\langle w^{-1}\rangle$ and $\alpha\beta\in pA\langle w^{-1}\rangle$. Since $A\langle z^{-1}\rangle$ is fc in $A\langle w^{-1}\rangle$ by (2), $p$ is a prime element in $A\langle w^{-1}\rangle$. Therefore, $\alpha\in pA\langle w^{-1}\rangle\subset p\A_1$ or $\beta\in pA\langle w^{-1}\rangle\subset p\A_1$ and hence $p$ is prime in $\A_1$. 

\medskip
{\bf (4)} For $(a,b)\in \mathcal{P}_A$, let $z=ax+by$. Then the following equation holds in $A\langle z^{-1}\rangle$: 
\[
a\cdot\dfrac{\:x\:}{z}+b\cdot\dfrac{\:y\:}{z}=\dfrac{\:ax+by\:}{z}=1. 
\]
Therefore, $(a,b)$ is comaximal in $A\langle z^{-1}\rangle$, and hence it is also comaximal in $\A_1$. 
\end{proof} 

For $n\geq1$, let $\A_{n+1}:=\varinjlim \A_n\langle z^{-1}\rangle$ be the direct limit of the direct system $(\A_n\langle z^{-1}\rangle,\varphi_{wz})$, where $\varphi_{wz}:\A_n\langle z^{-1}\rangle\to\A_n\langle w^{-1}\rangle$ is defined by the same way for $z,w\in \mathfrak{S}_{\A_n}$ with $z<w$. Then we have the following chain of UFDs in which each inclusion is fc: 
\[
A:=\A_0\subset \A_1\subset \cdots \subset \A_n\subset \cdots. 
\]

\begin{lem} Let $B:=\bigcup_{n\geq0}\A_n$. Then $B$ is a PID and $A$ is fc in $B$. 
\end{lem}
\begin{proof}
$B$ can be regarded as the direct limit of the direct system defined by the inclusions $\A_i\subset \A_{i+1}$. By the construction, each $\A_i$ is a UFD and any two relatively prime elements in $\A_i$  are comaximal in $\A_{i+1}$. It follows from \cite[Lemma 2.4]{CGM19} that $B$ is a PID and $A$ is fc in $B$. 
\end{proof}

\subsection{A qfc subalgebra of a PID}

Question \ref{CGM} does not hold even under the qfc condition.  
The following examples shows that a qfc subalgebra of a PID (which is not a field) is not necessarily a PID. 

\begin{example} \label{non-normal}
{\rm
Let $k$ be a field and $A=k[x^2,x^3]\subset k[x^{\pm1}]=B$. By Theorem \ref{one var} (or Theorem \ref{d-th root}), $A$ is qfc in $B$. 
On the other hand, although $B$ is a PID, $A$ is not a PID since it is not normal (and hence not ac in $B$). 
Moreover, it follows from Corollary \ref{d-th root pfc} that $A$ is not pfc in $B$. 
}
\end{example}

\medskip

\noindent {\bf Acknowledgments.} The authors wish to thank Professor Gene Freudenburg of Western Michigan University and 
Professor Hideo Kojima of Niigata University for their helpful comments regarding this paper. 



\begin{thebibliography}{30}
\bibitem{AAZ92}
D. D. Anderson, D. F. Anderson and M. Zafrullah, 
Factorization in integral domains. II, 
J. Algebra {\bf 152} (1992), no.1, 78--93.

\bibitem{AC22}
D. D. Anderson and S. Chun, 
Inert type extensions and related factorization properties, 
J. Algebra Appl. {\bf 21} (2022), no.7, Paper No. 2350177, 27 pp.

\bibitem{BH98}
W. Bruns and J. Herzog, 
{\it Cohen-Macaulay rings}, 
revised ed., Cambridge Stud. Adv. Math., vol. 39, Cambridge University Press, Cambridge, 1998.

\bibitem{CGM15}
S. Chakraborty, R. V. Gurjar and M. Miyanishi, 
Factorially closed subrings of commutative rings, 
Algebra Number Theory {\bf 9} (5) (2015) 1137--1158.  
w
\bibitem{CGM19} 
S. Chakraborty, R. V. Gurjar and D. Mondal, 
Quasi-factorially closed subrings of commutative rings, 
J. Algebra {\bf 534} (2019), 190-206.

\bibitem{Coh68}
P. M. Cohn, 
Bezout rings and their subrings, 
Proc. Cambridge Philos. Soc. {\bf 64} (1968), 251--264.


\bibitem{Cos77}
D. Costa, 
Retracts of polynomial rings, 
J. Algebra {\bf 44} (1977) 492--502. 


\bibitem{GN}
N. Gupta and T. Nagamine,  
Retracts of Laurent polynomial rings, 
Journal of the Ramanujan Mathematical Society {\bf 40} (4) (2025), 379--383. 

\bibitem{Mon21}
D. Mondal, 
Divisor closed and quasi-divisor closed extensions, 
Comm. Algebra {\bf 49} (2021), no.8, 3477--3489.

\bibitem{RG}
J.C. Rosales and P. A. Garc\'{\i}a-S\'{a}nchez, {\it Numerical Semigroups}, Developments in Mathematics, Springer New York,  2009. 

\bibitem{Zie95}
G\"unter M. Ziegler, {\it Lectures on Polytopes}, Graduate Texts in Mathematics, vol. 152, Springer-Verlag, New York, 1995.
\end{thebibliography}
\end{document}